\documentclass[12pt]{amsart}
\usepackage{latexsym, amsmath, amscd, amssymb, amsthm} 
\usepackage[T2A]{fontenc}
 \usepackage[cp1251]{inputenc}
\usepackage[english]{babel}
\newtheorem{te}{Theorem}
 
 \newtheorem{pr}{Example}
 \newtheorem{lm}{Lemma}

 \def\mathbi#1{\textbf{\em #1}}

\begin{document}

\noindent

 \title[]{ Some new identities for Schur polynomials }

\author{Leonid Bedratyuk}\address{Khmelnitskiy national university, Instituts'ka, 11,  Khmelnitskiy, 29016, Ukraine}
\email{LeonidBedratyuk@khmnu.edu.ua}

\begin{abstract} For the Schur polynomials  bounded and unbounded generalizations of the Cauchy identities are  found. 

\end{abstract}
\maketitle
\section{Introduction}

Let  $\mathcal{P}_n $ be  the set of all partitions with length at most $n.$  A partition $\lambda = (\lambda_1,\lambda_2, \ldots, \lambda_n)$  is a tuple  of nonnegative
integers  which are called parts  and which are ordered as $\lambda_i \geq \lambda_{i+1}.$ The sum of the entries is denoted $|\lambda|= \lambda_1+\lambda_2+ \cdots + \lambda_n$.
Let us introduce a partial ordering $\leq$ in   $\mathcal{P}_n $ by defining  $\lambda \leq \mu $ if $\lambda_i \leq \mu_i$ for all $i.$

The Schur polynomial  corresponding to   $\lambda \in \mathcal{P}_n$  is defined as the following polynomial  in variables
$\boldsymbol{x}=(x_1,x_2, \ldots, x_n):$ 

$$
\boldsymbol{s}_\lambda(\boldsymbol{x})=\frac{\det(x_j^{\lambda_i+n-i})}{\det(x_j^{n-i})}=\frac{\begin{vmatrix} x_1^{\lambda_1+n-1} & x_2^{\lambda_1+n-1} & \ldots & x_n^{\lambda_1+n-1} \\  x_1^{\lambda_2+n-2} & x_2^{\lambda_2+n-2} & \ldots & x_n^{\lambda_2+n-2} \\ 
\vdots & \vdots & \ldots & \vdots \\
x_1^{\lambda_n} & x_2^{\lambda_n} & \ldots & x_n^{\lambda_n}
\end{vmatrix}}{\begin{vmatrix} x_1^{n-1} & x_2^{n-1} & \ldots & x_n^{n-1} \\  x_1^{n-2} & x_2^{n-2} & \ldots & x_n^{n-2} \\ 
\vdots & \vdots & \ldots & \vdots \\
1 & 1 & \ldots & 1
\end{vmatrix}}.
$$

The total degree of the Schur polynomial  $\boldsymbol{s}_\lambda(\boldsymbol{x})$ equals $|\lambda|.$

 The Schur polynomials are invariants of the symmetric group  $S_n$.  One can read more about symmetric functions and symmetric polynomials in the classic books   \cite{St}, \cite{Mac}.

The famous  Cauchy identities  for Schur polynomials  state that
   \begin{gather*}
\sum_{\lambda \in \mathcal{P}_n} \boldsymbol{s}_\lambda(\boldsymbol{x}) \boldsymbol{s}_\lambda(\boldsymbol{y})=\prod_{i=1}^n \prod_{j=1}^m \frac{1}{1-x_i y_j},\\
\sum_{\lambda \in \mathcal{P}_n} \boldsymbol{s}_\lambda(\boldsymbol{x}) \boldsymbol{s}_{\lambda'}(\boldsymbol{y})=\prod_{i=1}^n \prod_{j=1}^m {(1+x_i y_j)},
\end{gather*}
here $\boldsymbol{y}=(y_1,y_2, \ldots, y_m)$ and $\lambda'$ is the conjugate partition ( see \cite{Mac} for definition of  conjugate partition).
There are many different generalizations of these identities, for example, see  \cite{BW},\cite{IT}.

The aim of the paper  is to prove some  variants of the
Cauchy identities. Firstly, we prove that  for arbitrary family of polynomials  $f_i=f_i(\boldsymbol{y}), i=0, 1, \ldots$ the following identity holds  
$$
\sum_{{\lambda \in \mathcal{P}_n}} \boldsymbol{s}_{\lambda}(\boldsymbol{x}) \left | \begin{array}{lllll}
f_{\lambda_1} & f_{\lambda_1+1} & f_{\lambda_1+2} & \ldots & f_{\lambda_1+n-1}\\
f_{\lambda_2-1} & f_{\lambda_2} & f_{\lambda_2+1} & \ldots & f_{\lambda_2+n-2}\\
\vdots & \vdots  & \vdots  & \ldots & \vdots \\
f_{\lambda_n-(n-1)} & f_{\lambda_n-(n-2)} & f_{\lambda_n-(n-3)} & \ldots & f_{\lambda_n}
\end{array} \right|=\prod_{i=1}^n F(\boldsymbol{y},x_i),
$$
whеre  $$
F(\boldsymbol{y},z) =\sum_{i=0}^\infty  f_i z^i.
$$ 
We prove that this  identity implies the Cauchy identities.

 For the set  $t_1,t_2,\ldots, t_n$ indeterminates the following identity holds
\begin{gather*}
\sum_{\lambda \in \mathcal{P}_n}  \boldsymbol{s}_\lambda  \det(t_j^{\lambda_i+j-i})=\frac{\displaystyle \det\left(\frac{x_j^{n-i}}{1-x_j t_i}\right)}{\displaystyle \det(x_j^{n-i})}.
\end{gather*}
Here we take $f_i = 0$ and  $t_i=0$  for all $i < 0.$

Secondly, we prove bounded variants of the above two identities.
Macdonald \cite{Mac}  proved the following bounded analogue of the Littlewood identities \cite{Lit}:
  
$$
\sum_{ \substack{\lambda \in \mathcal{P}_n, \\ \lambda \leq (a,a,\ldots,a)}} \boldsymbol{s}_{\lambda}(\boldsymbol{x})=\frac{\det(x_i^{a+2n-j}-x_i^{j-1})}{\displaystyle \prod_{i=1}^n (x_i-1) \prod_{1 \leq i<j \leq n} (x_i-x_j) (x_i x_j-1)}.
$$

Let  
$$
\sum_{i=0}^{a} f_i z^i=F(\boldsymbol{y},z,a), a \in \mathbb{N}.
$$

Then we prove that 
\begin{gather*}
\sum_{\substack{\lambda \in \mathcal{P}_n \\ \lambda \leq (a_1,a_2, \ldots,a_n)} } \boldsymbol{s}_{\lambda}(\boldsymbol{x}) \det(f_{\lambda_i-i+j})=\frac{\det\left(x_j^{n-i} F(\boldsymbol{y},x_i,a_i)\right)}{\det(x_j^{n-i})},
\end{gather*}

and

\begin{gather*}
\sum_{\substack{\lambda \in \mathcal{P}_n \\ \lambda \leq (a_1,a_2, \ldots,a_n)}} \boldsymbol{s}_{\lambda }  \det(t_j^{\lambda_i+j-i})=\frac{\displaystyle \det\left(\frac{x_j^{n-i}(1-(x_j t_i)^{a_i+1})}{1-x_i t_j}\right)}{\det(x_j^{n-i})}. 
\end{gather*}

 If  $\lambda_i+j-i<0$ or $\lambda_i+j-i > m_i$ then corresponding entries of the matrices  $(f_{\lambda_i+j-i})$ and  $(t_j^{\lambda_i+j-i})$  are equal to $0.$

We will also consider some specifications of these identities.

\section{Unbounded identities}

The symmetric group $S_n$ acts on $\mathbb{N}^n$ by by permuting the coordinates.
Let us define  a binary relation $\sim$  on the set $\mathbb{N}^n$ 
 as follows:
for  $\mu, \nu \in \mathbb{N}^n$ we have    $\mu \sim \nu, $    if there exists a permutation  $\pi_\mu \in S_n$ such that  
$\mu+\delta_n=\pi_\mu(\nu+\delta_n)$. Here  $\delta_n=(n-1,n-2, \ldots, 1,0).$
It is easy to show that the relation $\sim$ is an equivalence relation on $\mathbb{N}^n.$
Then  $\mathbb{N}^n$  is a  union of mutually disjoint equivalence classes $[\mu]_\sim$
$$
\mathbb{N}^n=\bigcup_{\mu \in \mathbb{N}^n/\sim } [\mu]_\sim,
$$
where  $\mu$ runs the set of representatives of the equivalence relation $\sim$.

Let  $f_i=f_i(\boldsymbol{y}), i=0,1, \ldots$ be a polynomial family and for $\mu \in \mathbb{N}^n$ we denote  $f_\mu=f_{\mu_1} f_{\mu_2}  \cdots f_{\mu_n}.$

The following lemma plays important role in the proof of our main results.

\begin{lm}\label{l2}  Suppose that  for  $\mu,\lambda \in \mathbb{N}^n$  we have    $\mu \sim \lambda$ and  $\lambda$  is a partition. Then 

$$
 \sum_{\mu \in [\lambda]_\sim} {\rm sgn}(\pi_\mu)  f_{\mu}=\det(f_{\lambda_i-i+j}).
$$
\end{lm}
\begin{proof}
Since  $\mu=\pi_\mu(\lambda+\delta_n)-\delta_n$ we get
$$\mu_j=\lambda_{\pi_\mu(j)}+(n-\pi_\mu(j))-(n-j)=\lambda_{\pi_\mu(j)}-\pi_\mu(j)+j.$$
Then 
$$
 \sum_{\mu \in [\lambda]_\sim} {\rm sgn}(\pi_\mu)  f_{\mu} =\sum_{\pi \in S_n} {\rm sgn}(\pi)   f_{\pi(\lambda+\delta_n)-\delta_n}=\sum_{\pi\in S_n} {\rm sgn}(\pi) \prod_{j=1}^n f_{\lambda_{\pi_\mu(j)}-\pi_\mu(j)+j}=\det(f_{\lambda_i-i+j}).
$$
For the case  $\lambda_i-i+j<0$ the corresponding entries of the determinant should be equals to zero.
\end{proof}

Let us consider an example

\begin{pr}{\rm   For  $n=3$ and for a partition $\lambda=(\lambda_1,\lambda_2, \lambda_3)$
we have

\begin{center}
\begin{tabular}{|l|c|c|}
\hline
$\mu=\pi_\mu(\lambda+\delta_n)-\delta_n$ & $\pi_\mu$  & ${\rm sgn} (\pi_\mu)$ \\ \hline
  $ (\lambda_1,\lambda_2, \lambda_3) $               &            $e$            &                 $ +1$                      \\ \hline
  $ (\lambda_{{1}},\lambda_{{3}}-1,\lambda_{{2}}+1) $                &   $(23)$                    &     $-1$                                  \\ \hline
   $(\lambda_{{2}}-1,\lambda_{{1}}+1,\lambda_{{3}} )$                &    $(12)$                    &    $-1$                                    \\ \hline
  $(\lambda_{{2}}-1,\lambda_{{3}}-1,\lambda_{{1}}+2)$                 &   $(123)$                     &     $+1$                                   \\ \hline
   $(\lambda_{{3}}-2,\lambda_{{1}}+1,\lambda_{{2}}+1)$                &     $(132)$                   &       $+1$                                 \\ \hline
	 $(\lambda_{{3}}-2,\lambda_{{2}},\lambda_{{1}}+2)$                &     $(13)$                   &       $-1$                                 \\ \hline
\end{tabular}
\end{center}
Then 
\begin{gather*}
\sum_{\pi_\mu \in S_n} {\rm sgn}(\pi_\mu)   f_{\pi_\mu(\lambda+\delta_n)-\delta_n}=
f_{{\lambda_{{1}}}}f_{{\lambda_{{2}}}}f_{{\lambda_{{3}}}}{-}f_{{\lambda_
{{1}}}}f_{{\lambda_{{3}}{-}1}}f_{{\lambda_{{2}}+1}}{-}f_{{\lambda_{{3}}{-}2}
}f_{{\lambda_{{2}}}}f_{{\lambda_{{1}}+2}}{-}f_{{\lambda_{{2}}{-}1}}f_{{
\lambda_{{1}}+1}}f_{{\lambda_{{3}}}}+\\+f_{{\lambda_{{3}}{-}2}}f_{{\lambda_
{{1}}+1}}f_{{\lambda_{{2}}+1}}+f_{{\lambda_{{2}}{-}1}}f_{{\lambda_{{3}}{-}
1}}f_{{\lambda_{{1}}+2}}=\left | \begin{array}{lll}
f_{\lambda_1} & f_{\lambda_1+1} & f_{\lambda_1+2} \\
f_{\lambda_2-1} & f_{\lambda_2} & f_{\lambda_2+1}\\
f_{\lambda_3-2} & f_{\lambda_3-1} & f_{\lambda_3} 
\end{array} \right|=\det(f_{\lambda_i-i+j}).
\end{gather*}

}

\end{pr}

Note that the Lemma  remains true and for finite sequence of polynomials 
  $f_0, f_1, \ldots, f_a$,$a \in \mathbb{N}$. Then for the case  $\lambda_i-i+j>a$ the  corresponding entries of the determinant $\det(f_{\lambda_i-i+j})$ should be equal  to zero.

Now, for any  $\mu \in \mathbb{N}^n$ define the polynomial

$$
\mathbi{S}_\mu(\mathbi{x})=\frac{\det(x_j^{\mu_i+n-i})}{\det(x_j^{n-i})}.
$$

The following Lemma states that  $\mathbi{S}_\mu(\mathbi{x})$  either equals zero or up to sign is equal to the Schur polynomial  $\boldsymbol{s}_\lambda(\boldsymbol{x})$ where the partition  $\lambda$ is a representative of equivalence class to which $\mu$ belongs.

\begin{lm}\label{l1}  If  $\mu \in \mathbb{N}^n$ is equivalent to  $ \lambda \in \mathcal{P}_n$ and  $\pi_\mu$ be the corresponding permutation; then 
$$ 
\boldsymbol{S}_\mu(\boldsymbol{x})={\rm sgn }(\pi_\mu) \boldsymbol{s}_\lambda(\boldsymbol{x}).
$$
If  $\mu$ is not equivalent to any partition then  
 $\boldsymbol{S}_\mu(\boldsymbol{x})=0.$
\end{lm}
\begin{proof}
For a partition  $\lambda$  the sum  $\lambda+\delta_n$  consists of a  strong decreasing sequence of integers
$$
\lambda_1+n-1 > \lambda_2+n-2 > \cdots > \lambda_n.
$$
Since  $\mu+\delta_n$  obtained from  $\lambda+\delta_n$  by permutation   $\pi_\mu$ then  the 
corresponding determinants are equal up to a sign:
$$
\det(\mu+\delta_n)={\rm sgn}(\pi_\mu) \det(\lambda+\delta_n) \neq 0.
$$ 
Dividing the both sides by  $\det(x_j^{n-i})$  we get    

$$ 
\boldsymbol{S}_\mu(\boldsymbol{x})={\rm sgn }(\pi_\mu) \boldsymbol{s}_\lambda(\boldsymbol{x}),
$$
as required.

Suppose now that $\mu$ does not equivalent to any partition  $\lambda.$ Then   $\mu+\delta_n$  has  two equal parts and   it implies  that   $\det(\mu+\delta_n)=0.$
\end{proof}

The following theorem is the main result of the section.

\begin{te}\label{t1} 

\begin{enumerate}
	\item 

Let  $f_i=f_i(\boldsymbol{y})$ be a polynomial family and

$$
\sum_{i=0}^\infty  f_i z^i=F(\boldsymbol{y},z).
$$ 

Then the following identity holds
$$
\sum_{{\lambda \in \mathcal{P}_n}} \boldsymbol{s}_{\lambda}(\boldsymbol{x}) \left | \begin{array}{lllll}
f_{\lambda_1} & f_{\lambda_1+1} & f_{\lambda_1+2} & \ldots & f_{\lambda_1+n-1}\\
f_{\lambda_2-1} & f_{\lambda_2} & f_{\lambda_2+1} & \ldots & f_{\lambda_2+n-2}\\
\vdots & \vdots  & \vdots  & \ldots & \vdots \\
f_{\lambda_n-(n-1)} & f_{\lambda_n-(n-2)} & f_{\lambda_n-(n-3)} & \ldots & f_{\lambda_n}
\end{array} \right|=\prod_{i=1}^n F(\boldsymbol{y},x_i).
$$

\item Let  $t_1,t_2,\ldots, t_n$ be some set of variables. Then  
 
\begin{gather*}
\sum_{\lambda \in \mathcal{P}_n}  \boldsymbol{s}_\lambda  \det(t_j^{\lambda_i+j-i})=\frac{\displaystyle \det\left(\frac{x_j^{n-i}}{1-x_j t_i}\right)}{\displaystyle \det(x_j^{n-i})}.
\end{gather*}

\end{enumerate}
\end{te}
Here we take $f_i = 0$ and  $t_i$  for all $i < 0.$

\begin{proof} $(i)$
Lemma~\ref{l2} and Lemma~\ref{l1} implies 
$$
\sum_{\mu \in [\lambda]}  \boldsymbol{S}_\mu(\boldsymbol{x}) f_\mu= \boldsymbol{s}_\lambda(\boldsymbol{x})\left | \begin{array}{lllll}
f_{\lambda_1} & f_{\lambda_1+1} & f_{\lambda_1+2} & \ldots & f_{\lambda_1+n-1}\\
f_{\lambda_2-1} & f_{\lambda_2} & f_{\lambda_2+1} & \ldots & f_{\lambda_2+n-2}\\
\vdots & \vdots  & \vdots  & \ldots & \vdots \\
f_{\lambda_n-(n-1)} & f_{\lambda_n-(n-2)} & f_{\lambda_n-(n-3)} & \ldots & f_{\lambda_n}
\end{array} \right|
$$

Now on the one hand, we have 

$$
\sum_{\mu \in  \mathbb{N}^n} \boldsymbol{S}_\mu(\boldsymbol{x}) f_\mu=\sum_{\lambda \in  \mathcal{P}_n} \boldsymbol{s}_\lambda(\boldsymbol{x}) \left | \begin{array}{lllll}
f_{\lambda_1} & f_{\lambda_1+1} & f_{\lambda_1+2} & \ldots & f_{\lambda_1+n-1}\\
f_{\lambda_2-1} & f_{\lambda_2} & f_{\lambda_2+1} & \ldots & f_{\lambda_2+n-2}\\
\vdots & \vdots  & \vdots  & \ldots & \vdots \\
f_{\lambda_n-(n-1)} & f_{\lambda_n-(n-2)} & f_{\lambda_n-(n-3)} & \ldots & f_{\lambda_n}
\end{array} \right|
$$

On the other hand, we can find  this sum explicitly

\begin{gather*}
\det(x_i^{n-j}) \sum_{\mu \in \mathbb{N}^n} \mathbi{S}_\mu f_{\mu} =
\sum_{\mu_1,\cdots,\mu_n=0}^\infty \begin{vmatrix}
 x_1^{\mu_1+n-1} &  x_2^{\mu_1+n-1} & \ldots & x_n^{\mu_1+n-1} \\
 x_1^{\mu_2+n-2} &  x_2^{\mu_2+n-2} & \ldots & x_n^{\mu_2+n-2} \\
\hdotsfor{4}\\
x_1^{\mu_n} &  x_2^{\mu_n} & \ldots & x_n^{\mu_n} \\
\end{vmatrix} f_{\mu_1} f_{\mu_2} \cdots f_{\mu_n}=\\=
\begin{vmatrix}
 \displaystyle \sum_{\mu_1=0}^\infty   x_1^{\mu_1+n-1} f_{\mu_1}   &   \displaystyle \sum_{\mu_1=0}^\infty   x_2^{\mu_1+n-1} f_{\mu_1} & \displaystyle \ldots & \displaystyle  \sum_{\mu_1=0}^\infty   x_n^{\mu_2+n-1} f_{\mu_1} \\
  \displaystyle \sum_{\mu_2=0}^\infty   x_1^{\mu_2+n-2} f_{\mu_2} &  \displaystyle  \sum_{\mu_2=0}^\infty   x_2^{\mu_2+n-2} f_{\mu_2} &  \displaystyle \ldots & \displaystyle  \sum_{\mu_2=0}^\infty   x_n^{\mu_2+n-2} f_{\mu_2}\\
\vdots & \vdots  & \ldots  & \vdots  \\\\
 \displaystyle \sum_{\mu_n=0}^\infty   x_1^{\mu_n} f_{\mu_n} &   \displaystyle  \sum_{\mu_n=0}^\infty   x_2^{\mu_n} f_{\mu_n} &  \displaystyle \ldots &  \displaystyle  \sum_{\mu_n=0}^\infty   x_n^{\mu_n} f_{\mu_n} \\
\end{vmatrix} =\\=\left | \begin{array}{lllll} 
x_1^{n-1}F(\boldsymbol{y},x_1) & x_2^{n-1}F(\boldsymbol{y},x_2)  & \ldots & x_n^{n-1}F(\boldsymbol{y},x_n)\\
x_1^{n-2}F(\boldsymbol{y},x_1) & x_2^{n-2}F(\boldsymbol{y},x_2)  & \ldots & x_n^{n-2}F(\boldsymbol{y},x_n)\\
\vdots & \vdots  & \vdots  & \ldots  \\
F(\boldsymbol{y},x_1) & F(\boldsymbol{y},x_2) &  \ldots & F(\boldsymbol{y},x_n)
\end{array} \right|=\det(x_i^{n-j}) \prod_{j=1}^n F(\boldsymbol{y},x_j).
\end{gather*}

Therefore  

$$
\sum_{\lambda \in  \mathcal{P}_n} \boldsymbol{s}_\lambda(\boldsymbol{x}) \left | \begin{array}{lllll}
f_{\lambda_1} & f_{\lambda_1+1} & f_{\lambda_1+2} & \ldots & f_{\lambda_1+n-1}\\
f_{\lambda_2-1} & f_{\lambda_2} & f_{\lambda_2+1} & \ldots & f_{\lambda_2+n-2}\\
\vdots & \vdots  & \vdots  & \ldots & \vdots \\
f_{\lambda_n-(n-1)} & f_{\lambda_n-(n-2)} & f_{\lambda_n-(n-3)} & \ldots & f_{\lambda_n}
\end{array} \right|=\prod_{j=1}^n F(\boldsymbol{y},x_j).
$$

Ley us prove  the part $(2)$ of the theorem.

By a similar argument  we have     
 \begin{gather*}
\sum_{\mu \in \mathbb{N}^n} \mathbi{S}_\mu \boldsymbol{t}^{\mu}=\sum_{\lambda \in \mathcal{P}_n} \boldsymbol{s}_\lambda \left | \begin{array}{lllll}
t_1^{\lambda_1} & t_2^{\lambda_1+1} & t_3^{\lambda_1+2} & \ldots & t_n^{\lambda_1+n-1}\\
t_1^{\lambda_2-1} & t_2^{\lambda_2} & t_3^{\lambda_2+1} & \ldots & t_n^{\lambda_2+n-2}\\
\vdots & \vdots  & \vdots  & \ldots & \vdots \\
t_1^{\lambda_n-(n-1)} & t_2^{\lambda_n-(n-2)} & t_3^{\lambda_n-(n-3)} & \ldots & t_n^{\lambda_n}
\end{array} \right|.
 \end{gather*}

On the other hand

\begin{gather*}
\det(x_j^{n-i}) \sum_{\mu \in \mathbb{N}^n} \mathbi{S}_\mu \boldsymbol{t}^{\mu}=\sum_{\lambda_1,\cdots,\lambda_n=0}^\infty \begin{vmatrix}
 x_1^{\lambda_1+n-1} &  x_2^{\lambda_1+n-1} & \ldots & x_n^{\lambda_1+n-1} \\
 x_1^{\lambda_2+n-2} &  x_2^{\lambda_2+n-2} & \ldots & x_n^{\lambda_2+n-2} \\
\hdotsfor{4}\\
x_1^{\lambda_n} &  x_2^{\lambda_n} & \ldots & x_n^{\lambda_n} \\
\end{vmatrix} t_1^{\lambda_1} t_2^{\lambda_2} \cdots t_n^{\lambda_n}=\\=
\begin{vmatrix}
 \sum\limits_{\lambda_1=0}^\infty x_1^{\lambda_1+n-1} t_1^{\lambda_1} &  \sum\limits_{\lambda_1=0}^\infty x_2^{\lambda_1+n-1} t_1^{\lambda_1} & \ldots & \sum\limits_{\lambda_1=0}^\infty  x_n^{\lambda_1+n-1} t_1^{\lambda_1} \\
 \sum\limits_{\lambda_2=0}^\infty  x_1^{\lambda_2+n-2} t_2^{\lambda_2} &  \sum\limits_{\lambda_2=0}^\infty x_2^{\lambda_2+n-2} t_2^{\lambda_2} & \ldots & \sum\limits_{\lambda_2=0}^\infty  x_n^{\lambda_2+n-2} t_2^{\lambda_2}\\
\vdots &  \vdots & \ldots & \vdots \\
\sum\limits_{\lambda_n=0}^\infty  x_1^{\lambda_n} t_n^{\lambda_n} &  \sum\limits_{\lambda_n=0}^\infty  x_2^{\lambda_n} t_n^{\lambda_n} & \ldots & \sum\limits_{\lambda_n=0}^\infty  x_n^{\lambda_n} t_n^{\lambda_n}
\end{vmatrix}  =\\=
\begin{vmatrix}
 \displaystyle \frac{x_1^{n-1}}{1-x_1 t_1} &   \displaystyle \frac{x_2^{n-1}}{1-x_2 t_1} & \displaystyle \ldots & \displaystyle   \frac{x_1^{n-1}}{1-x_1 t_1} \\
  \displaystyle \frac{x_1^{n-2}}{1-x_1 t_2} &  \displaystyle   \frac{x_2^{n-2}}{1-x_2 t_2} &  \displaystyle \ldots & \displaystyle   \frac{x_1^{n-2}}{1-x_1 t_2}\\
\hdotsfor{4}\\
 \displaystyle \frac{1}{1-x_1 t_n} &   \displaystyle  \frac{1}{1-x_2 t_n} &  \displaystyle \ldots &  \displaystyle  \frac{1}{1-x_1 t_n} \\
\end{vmatrix}=\det\left(\frac{x_j^{n-i}}{1-x_j t_i}\right). 
\end{gather*}
Thus
\begin{gather*}
\sum_{\lambda \in \mathcal{P}_n} \boldsymbol{s}_\lambda(\boldsymbol{x})  \det(t_j^{\lambda_i-i+j}) =\frac{\displaystyle \det\left(\frac{x_j^{n-i}}{1-x_j t_i}\right)}{\displaystyle \det(x_j^{n-i})}.
\end{gather*}

\end{proof}

The theorem implies the two classical Cauchy identities. In fact, put $f_i=h_i$, where  $h_i$ is the complette symmetrical polynomial in the  $m$ variables $\boldsymbol{y}=(y_1,y_2, \ldots, y_m)$.
Taking into account the identity  
$$
\sum_{i=0}^\infty  h_i z^i=\prod_{i=0}^n \frac{1}{1-y_i z},
$$
the  Theorem~\ref{t1} implies
\begin{gather*}
\sum_{\lambda} \boldsymbol{s}_{\lambda}(\boldsymbol{x}) \left | \begin{array}{lllll}
h_{\lambda_1} & h_{\lambda_1+1} & h_{\lambda_1+2} & \ldots & h_{\lambda_1+n-1}\\
h_{\lambda_2-1} & h_{\lambda_2} & h_{\lambda_2+1} & \ldots & h_{\lambda_2+n-2}\\
\vdots & \vdots  & \vdots  & \ldots & \vdots \\
h_{\lambda_n-(n-1)} & h_{\lambda_n-(n-2)} & h_{\lambda_n-(n-3)} & \ldots & h_{\lambda_n}
\end{array} \right|=\prod_{i=1}^n \prod_{j=1}^m \frac{1}{1-x_i y_j}.
\end{gather*}

By using the Jacobi-Trudi identity
$$
\boldsymbol{s}_\lambda(\boldsymbol{y})=\det(h_{\lambda_i-i+j}),
$$
we get the Cauchy identity:

$$
\sum_{\lambda \in \mathcal{P}_n} \boldsymbol{s}_\lambda(\boldsymbol{x}) \boldsymbol{s}_\lambda(\boldsymbol{y})=\prod_{i=1}^n \prod_{j=1}^m \frac{1}{1-x_i y_j}.
$$

Similarly, by  using the dual Jacobi-Trudi identity 

$$
\boldsymbol{s}_{\lambda'}(\boldsymbol{y})=\det(e_{\lambda_i-i+j})
$$

we obtain the dual  Cauchy identity
$$
\sum_{\lambda \in \mathcal{P}_n} \boldsymbol{s}_\lambda(\boldsymbol{x}) \boldsymbol{s}_{\lambda'}(\boldsymbol{y})=\prod_{i=1}^n \prod_{j=1}^m {(1+x_i y_j)}.
$$

Let us consider some specialization of the formulas.

\begin{pr}{\rm
Put $f_i=p_i$ where  $p_i$ is the power sum polynomial. Then, taking into account

\begin{gather*}
\sum_{i=0}^\infty  p_i z^i=\sum_{i=0}^n \frac{1}{1-x_i z},
\end{gather*}
we obtain
\begin{gather*}
\sum_{\lambda \in \mathcal{P}_n} \boldsymbol{s}_{\lambda}(\boldsymbol{x}) \det(p_{\lambda_i-i+j})=  \prod_{j=1}^m \sum_{i=1}^i \frac{1}{1-x_i y_j}.
\end{gather*} 

}
\end{pr}

\begin{pr}{\rm    Put $f_i=e_i+h_i.$  Then 

$$
\sum_{\lambda} \boldsymbol{s}_{\lambda}(\boldsymbol{x}) \det(e_{\lambda_i-i+j}+h_{\lambda_i-i+j})=\prod_{i=1}^n \left(\prod_{j=1}^m ({1+x_i y_j})+\prod_{j=1}^m \frac{1}{1-x_i y_j}  \right).
$$

}
\end{pr}

\begin{pr}{\rm 

Put   $t_1=t_2=\cdots=t_n=z.$ It is easy to see that for  $\lambda_2\neq 0$ the first two rows of the determinant  $\det(z^{\lambda_i-i+j})$ are proportional. Then for all partitions  except   $\lambda=(n)$ it equals to zero. For  $\lambda=(n)$ we get  upper-triangular
matrices with the main diagonal $(z^n,1,1,\ldots,1).$ Thus 
$$
\det(z^{\lambda_i-i+j})=\det({\rm diag}(z^n,1,1,\ldots,1))=z^n.
$$
Since  $\boldsymbol{s}_{(n)}(\boldsymbol{x})=h_i$, then 
\begin{gather*}
\sum_{i=0}^\infty h_i z^i   =\prod_{i=0}^n \frac{1}{1-x_i z}.
\end{gather*}

}
\end{pr}


\section{Bounded identities}


For  bounded sums the following theorem holds

\begin{te}  

\vspace{0.5cm}
 \begin{enumerate}
	 \item 

Let  
$$
\sum_{i=0}^{a} f_i z^i=F(\boldsymbol{y},z,a), a \in \mathbb{N}.
$$

Then 
\begin{gather*}
\sum_{\substack{\lambda \in \mathcal{P}_n \\ \lambda \leq (a_1,a_2, \ldots,a_n)} } \boldsymbol{s}_{\lambda}(\boldsymbol{x}) \det(f_{\lambda_i-i+j})=\frac{\det\left(x_j^{n-i} F(\boldsymbol{y},x_i,a_i)\right)}{\det(x_j^{n-i})}.
\end{gather*}

\item \begin{gather*}
\sum_{\substack{\lambda \in \mathcal{P}_n \\ \lambda \leq (a_1,a_2, \ldots,a_n)}} \boldsymbol{s}_{\lambda }  \det(t_j^{\lambda_i+j-i})=\frac{\displaystyle \det\left(\frac{x_j^{n-i}(1-(x_j t_i)^{a_i+1})}{1-x_i t_j}\right)}{\det(x_j^{n-i})}. 
\end{gather*}

 \end{enumerate} 
Here $f_{\lambda_i+j-i}$ and  $t_j^{\lambda_i+j-i}$  is equal to $0$  if  $\lambda_i+j-i<0$ or $\lambda_i+j-i > a_i.$ 
\end{te}

\begin{proof}
$(1)$   We have 

$$
\sum_{\substack{\mu \in  \mathbb{N}^n \\  \mu \leq (a_1,a_2, \ldots,a_n) }} \boldsymbol{S}_\mu(\boldsymbol{x}) f_\mu=\sum_{\lambda \in  \mathcal{P}_n} \boldsymbol{s}_\lambda(\boldsymbol{x}) \det(f_{\lambda_i-i+j}).
$$
On the other hand, we can find  this sum explicitly

\begin{gather*}
\det(x_j^{n-i}) \sum_{\substack{\mu \in \mathbb{N}^n \\ \mu \leq (a_1,a_2,\ldots,a_n)} } \mathbi{S}_\mu f_{\mu} =\sum_{\mu_1=0}^{a_1} \sum_{\mu_2=0}^{a_2}\cdots \sum_{\mu_n=0}^{a_n} 
 \det(x_j^{\mu_i+n-i}) f_{\mu_1} f_{\mu_2} \cdots f_{\mu_n}=\\=
\sum_{\mu_1=0}^{a_1} \sum_{\mu_2=0}^{a_2}\cdots \sum_{\mu_n=0}^{a_n} \det(x_j^{\mu_i+n-i} f_{\mu_i}) = \det\left(\sum_{\mu_i=0}^{a_i} x_j^{\mu_i+n-i} f_{\mu_i}\right)=\det\left(x_j^{n-i} F(\boldsymbol{y},x_i,a_i)\right).
\end{gather*}

Thus

\begin{gather*}
\sum_{\lambda \leq (a_1,a_2,\ldots,a_n) } \boldsymbol{s}_{\lambda}(\boldsymbol{x}) \det(f_{\lambda_i-i+j})=\frac{\det\left(x_j^{n-i} F(\boldsymbol{y},x_i,a_i)\right)}{\det(x_j^{n-i})}.
\end{gather*}
as required. 

$(2)$   In the same way  

\begin{gather*}
\det(x_j^{n-i}) \sum_{\substack{\lambda \in \mathcal{P}_n \\\lambda_i\leq a_i}} \mathbi{S}_\lambda t_1^{\lambda_1} \cdots t_{n}^{\lambda_{n}}=\sum_{\substack{\lambda \mathcal{P}_n \\\lambda_i\leq a_i}}  \det(x_j^{\lambda_i+n-i}) t_1^{\lambda_1} t_2^{\lambda_2} \cdots t_n^{\lambda_n}=\\=
\sum_{\substack{\lambda \\\lambda_i\leq a_i}}  \det(x_j^{\lambda_i+n-i} t_i^{\lambda_i}) =\det( \sum_{\lambda_i=0}^{a_i} x_j^{\lambda_i+n-i} t_i^{\lambda_i})=
\det\left( \frac{x_j^{n-i}(1-(x_j t_i)^{a_i+1})}{1-x_j t_i}\right)
=\\=
\begin{vmatrix}
 \displaystyle \frac{x_1^{n-1}(1-(x_1 t_1)^{a_1+1})}{1-x_1 t_1} &   \displaystyle \frac{x_2^{n-1} (1-( x_2 t_1)^{a_1+1})}{1-x_2 t_1} & \displaystyle \ldots & \displaystyle   \frac{x_n^{n-1} (1-( x_n t_1)^{a_1+1})}{1-x_n t_1} \\
  \displaystyle \frac{x_1^{n-2}(1-(t_2 x_1)^{a_2+1})}{1-x_1 t_2} &  \displaystyle   \frac{x_2^{n-2}(1-(t_2 x_2)^{a_2+1})}{1-x_2 t_2} &  \displaystyle \ldots & \displaystyle   \frac{x_1^{n-2}(1-( x_n t_2)^{a_2+1})}{1-x_n t_2}\\
\hdotsfor{4}\\
 \displaystyle \frac{1-( x_1 t_n)^{a_n+1}}{1-x_1 t_n} &   \displaystyle  \frac{1-( x_2 t_n)^{a_n+1}}{1-x_2 t_n} &  \displaystyle \ldots &  \displaystyle  \frac{1-( x_n t_n)^{a_n+1}}{1-x_n t_n}
\end{vmatrix}. 
\end{gather*}

\end{proof}

For the case $a_1=a_2=\cdots=a_n=a$ we have 

$$
\frac{\det\left(x_j^{n-i} F(\boldsymbol{y},x_i,a)\right)}{\det(x_j^{n-i})}=\frac{\det\left(x_j^{n-i} \right) \prod\limits_{i=1}^n F(\boldsymbol{y},x_i,a)}{\det(x_j^{n-i})}=\prod_{i=1}^n F(\boldsymbol{y},x_i,a),
$$

and

\begin{gather*}
\sum_{\lambda \leq  (a,a, \ldots,a) } \boldsymbol{s}_{\lambda}(\boldsymbol{x}) \det(f_{\lambda_i-i+j})=\prod_{i=1}^n F(\boldsymbol{y},x_i,a).
\end{gather*}

\begin{pr}{\rm Put  $f_i=e_i$, where  $e_i$ is the elementary  symmetrical polynomial in  $y_1,y_2, \ldots, y_m.$
Since 
$$
F(z)=\sum_{i=0}^m e_i z^i=\prod_{i=1}^m (1+y_i z)
$$
we have 

\begin{gather*}
\sum_{\lambda \leq (m,m, \ldots,m) } \boldsymbol{s}_{\lambda}(\boldsymbol{x}) \det(e_{\lambda_i-i+j})=\prod_{i=1}^n F(x_i)=\prod_{i=1}^n \prod_{j=1}^m (1+ x_i y_j ).
\end{gather*}

Note, that here  $\det(e_{\lambda_i-i+j})$ does not equal to  $\boldsymbol{s}_{\lambda'}(\boldsymbol{y})$ becouse  $e_{\lambda_i-i+j}$ is zero if  $\lambda_i-i+j>m.$

}
\end{pr}

For the case  $a_1=a_2=\cdots=a_n=a$  and  $t_1=t_2=\cdots=t_n=t$  we  have 

$$
\frac{\displaystyle \det\left(\frac{x_j^{n-i}(1-(x_j t)^{a+1})}{1-x_j t}\right)}{\det(x_j^{n-i})} =\frac{\det\left(x_j^{n-i} \right)  \prod\limits_{j=1}^m \frac{1-(x_j t)^{a+1}}{1-x_j t}}{\det(x_j^{n-i})}=\prod\limits_{i=1}^n \frac{1-(x_j t)^{a+1}}{1-x_j t}.
$$

Thus

\begin{gather*}
\sum_{\lambda \leq (a,a,\ldots, a)} \boldsymbol{s}_{\lambda }  \det(t^{\lambda_i+j-i})=\prod\limits_{j=1}^n \frac{1-(x_j t)^{a+1}}{1-x_j t}.
\end{gather*} 

\begin{pr}{\rm Let  $n=2$ and  $a=2$.  Then there exists  $6$ partition  $\lambda$  of length  $2$  such that  $\lambda  \leq (2,2):$

$$
(0,0),(1,0),(1,1),(2,0),(2,1),(2,2).
$$
We have 
\begin{center}
\begin{tabular}{|l|c|c|}
\hline
$\lambda$ & $\det(t^{\lambda_i+j-i})$ & $\boldsymbol{s}_{\lambda }$ \\ \hline
  $ (0,0) $               & $ \begin{vmatrix} 
	                            1 &0 \\
															0 & 1
															\end{vmatrix}=1$ & 1
                               \\  \hline	 
  $ (1,0) $               &  $\begin{vmatrix} 
	                            t &t^{2} \\
															0 & 1
															\end{vmatrix}=t$ & $x_{{1}}+x_{{2}}$
                                  \\ \hline
 $ (1,1) $               &  $\begin{vmatrix} 
	                            t &t^{2} \\
															1 & t
															\end{vmatrix}=0$ & $x_{{2}}x_{{1}}$
                                  \\ \hline	
 $ (2,0) $               &  $\begin{vmatrix} 
	                            t^2 &0 \\
															0 & 1
															\end{vmatrix}=t^2$ & ${x_{{1}}}^{2}+x_{{2}}x_{{1}}+{x_{{2}}}^{2}$
                                  \\ \hline	
 $ (2,1) $               &  $\begin{vmatrix} 
	                            t^2 &0 \\
															1 & t
															\end{vmatrix}=t^3$  & $\left( x_{{1}}+x_{{2}} \right) x_{{2}}x_{{1}}$
                                  \\ \hline	
 $ (2,2) $               &  $\begin{vmatrix} 
	                            t^2 &0 \\
															t & t^2
															\end{vmatrix}=t^4$ &  ${x_{{1}}}^{2}{x_{{2}}}^{2}$
                                  \\ \hline																																			
\end{tabular}
\end{center}
Note that for  $\lambda_i+j-i<0$ or $\lambda_i+j-i>2$ the corresponding entries of the matrix  $(t^{\lambda_i+j-i})_{i,j}$ is equal to zero.
Then  
\begin{gather*}
\sum_{\lambda \leq (2,2)} \boldsymbol{s}_{\lambda }  \det(t^{\lambda_i+j-i})=1{+} \left( x_{{1}}{+}x_{{2}} \right) t{+} \left( {x_{{1}}}^{2}{+}x_{{2}}x_{{1
}}{+}{x_{{2}}}^{2} \right) {t}^{2}{+} \left( x_{{1}}{+}x_{{2}} \right) x_{{2
}}x_{{1}}{t}^{3}{+}{x_{{1}}}^{2}{x_{{2}}}^{2}{t}^{4}=\\=
 \left( 1+tx_{{1}}+{t}^{2}{x_{{1}}}^{2}\right) \left( 1+tx_{{2}}+{t}^{2}{x_{{2}}}^{2} \right) =\frac{1-(x_1 t)^3}{1-x_1 t} \cdot \frac{1-(x_2 t)^3}{1-x_2 t}.
\end{gather*}
}
\end{pr}

For the specialization  $a_1=a_2=\cdots=a_n=a$  and  $t_1=t_2=\cdots=t_n=1$ let us consider the matrix $(1^{\lambda_i+j-i})_{i,j}=c(\lambda,a)$:

$$
c_{i,j}(\lambda,a)=\begin{cases} 0, \text{  $\lambda_i+j-i<0$ or $\lambda_i+j-i>a$ }\\
1, \text{   otherwise} 
\end{cases}.
$$

Then 

\begin{gather*}
\sum_{\lambda \leq (a,a,\ldots, a)} \boldsymbol{s}_{\lambda }  \det(c(\lambda,a))=\prod_{i=1}^n \frac{1-x_i^{a+1}}{1-x_i}.
\end{gather*}

\begin{pr}{\rm Put  $a=1$. Then all partitions which are less than   $(1,1, \ldots,1)$ has the form  $$\varepsilon^{(k)}= (\underbrace{1,1,\ldots,1}_{n-k \text{   times}},\underbrace{0,0, \ldots,0}_{k \text{ times} }), $$
The matrices  $c(\varepsilon^{(k)},1)$  are  lower-triangular
matrices with the main diagonal $(1,1,1,\ldots,1).$ 
Since  for all $k$ $\det(c(\varepsilon^{(k)},1))=1$ we obtain
\begin{gather*}
\sum_{\lambda \leq (1,1,\ldots, 1)} \boldsymbol{s}_{\lambda } =\prod_{i=1}^n \frac{1-x_i^{2}}{1-x_i}=\prod_{i=1}^n (1+x_i).
\end{gather*}

The bounded Macdonald formula gives the same result.
}
\end{pr}

\section{Acknowledgements}
We would  like to thank the Armed Forces of Ukraine for providing security to perform this work. 

\end{document}